\documentclass[makeidx]{article}
\title{Eigenvalues upper bounds for the magnetic Schr\"{o}dinger operator}
\author{Bruno Colbois, Ahmad El Soufi\footnote{Our colleague and friend Ahmad El Soufi passed away on December 29, 2016}, Sa\"{i}d Ilias and Alessandro Savo}

\date{\today}
\usepackage{makeidx}
\usepackage{amssymb , amsmath, amsthm}
\usepackage{graphicx}
\newtheorem{defi}{Definition} 
\newtheorem{thm}[defi]{Theorem}

 \newtheorem{prop}[defi]{Proposition}
\newtheorem{lemme}[defi]{Lemma}
\newtheorem{cor}[defi]{Corollary}
 
\newcommand{\twosystem}[2]{\left\{\begin{aligned} &#1\\ &#2\end{aligned}\right.}

\newcommand{\nero}{\smallskip$\bullet\quad$\rm}

\newcommand{\matrice}{\begin{pmatrix}}
\newcommand{\ok}{\end{pmatrix}}

\newcommand{\scal}[2]{\langle{#1},{#2}\rangle}

\newcommand{\abs}[1]{\lvert{#1}\rvert}
\newcommand{\norm}[1]{\lVert{#1}\rVert}

\newcommand{\reals}{{\bf R}}
\newcommand{\sphere}[1]{{\bf S}^{#1}}
\newcommand{\real}[1]{{\bf R}^{#1}}

\newcommand{\bd}{\partial}

\newcommand{\R}{\mathbb R}

\renewcommand{\l}{\lambda}

\begin{document}

\maketitle
\begin{abstract} 
We study the eigenvalues $\lambda_k(H_{A,q})$  of the magnetic Schr\"{o}dinger operator $H_{A,q}$ associated with a magnetic potential  $A$ and a scalar potential $q$, on  a compact Riemannian manifold $M$,  with Neumann boundary conditions if $\partial M\ne\emptyset$. We obtain various  bounds on $\lambda_k(H_{A,q})$. 
Besides the dimension and the volume of the manifold, the geometric quantity which plays an important role in these estimates is the first eigenvalue $\lambda_{1,1}''(M)$ of he Hodge-de Rham Laplacian acting on co-exact 1-forms. In the 2-dimensional case, $\lambda_{1,1}''(M)$ is nothing but the first positive eigenvalue of the Laplacian acting on functions.    As for the  dependence of the bounds on the potentials, it brings into play the mean value of the scalar potential $q$, the $L^2$-norm of the magnetic field $B=dA$, and the distance, taken in $L^2$, between the harmonic component of $A$ and the subspace  of all 
 closed $1$-forms whose cohomology class is integral (that is, having integral flux around any loop).  In particular, this distance is zero when the first cohomology group $H^1(M,\reals)$ is trivial. 

\medskip



\end{abstract}
\large

\section{Introduction} \label{intro} Let $(M,g)$ be a compact Riemannian manifold  with smooth boundary $\bd M$,  if non empty.
Consider the trivial complex line bundle $M\times\bf C$ over $M$; its space of sections  can be identified with $C^{\infty}(M,\bf C)$, the space of smooth complex valued functions on $M$.  Given a smooth real 1-form $A$ on $M$ we define a connection $\nabla^A$ on $C^{\infty}(M,\bf C)$ as follows:
\begin{equation} \label{connection}
\nabla^A_Xu=\nabla_Xu-iA(X)u
\end{equation}
for all vector fields $X$ on $M$ and for all $u\in C^{\infty}(M,\bf C)$ (here $\nabla$ denotes the Levi-Civita connection of $(M,g)$). The operator
\begin{equation} \label{magnetic laplacian}
\Delta_A=(\nabla^A)^{\star}\nabla ^A
\end{equation}
is called the {\it magnetic Laplacian} associated to the magnetic potential $A$, and the smooth two form 
$$
B=dA
$$
is the  associated {\it  magnetic field}. In this paper, we are interested in magnetic Schr\"{o}dinger operators of the form
$$
H_{A,q}=\Delta_A+q
$$
where $q$ is a real valued continuous function on $M$. If $A=0$, $\Delta_A$ is simply the usual Laplacian $\Delta$ on $M$. Note that we have

\begin{equation}\label{maglap}
\Delta_A u = \Delta u +2i \scal{A}{du} + \left(\vert A\vert ^2 +i\ \mbox{div} A\right) u
\end{equation}

where ${\rm div}A$ (often denoted also $\delta A$) is the co-differential of $A$.

\medskip
If the boundary of $M$ is non empty, we will consider Neumann magnetic conditions, that is:
\begin{equation}\label{mneumann}
\nabla^A_Nu=0\quad\text{on}\quad\bd M,
\end{equation}
where $N$ denotes the inner unit normal. 
Then, it is well-known that $H_{A,q}$ is self-adjoint, and admits a discrete spectrum
$$
\l_1(H_{A,q})\le \l_2(H_{A,q}) \le ... \to \infty.
$$

Estimates of eigenvalues of such operators 
have received a great attention in the  last decades, especially in the case where the underlying manifold is a bounded Euclidean domain with Dirichlet boundary conditions (see for instance \cite{AFNN, BH, Er1, NNT}) or with Neumann boundary conditions (see  \cite{BCC, CS, ELMP, Er2, FH2, HHHO, LLPP, Sh}).

\smallskip
In this paper, we first give \emph{upper bounds} for the spectrum of $H_{A,q}$ in terms of the harmonic part of the potential $A$, the magnetic field $B$, the integral $\int_{M} q$ and the geometry of $M$. These estimates are compatible with the Weyl law, and they are deduced from the fact that we have the relation (see (\ref{comparaison}) for a proof)
\begin{equation}
\lambda_k(H_{A,q})\le \lambda_k(H_{0,\vert A\vert^2+q}) =\lambda_k(\Delta+\vert A\vert^2+q)
\end{equation}
where $\vert A\vert$ denote the pointwise norm of $A$. We will also focus on the  two first eigenvalues $\lambda_1(H_{A,q})$ and $\lambda_2(H_{A,q})$, where we can get more precise results. 

\smallskip
 In Theorems \ref{firsteigenvalue}, \ref{volconf} and \ref{main1}, we observe that the geometry of the underling manifold $(M,g)$ appears through the first nonzero eigenvalue $\lambda_{1,1}''$ of the Hodge-De Rham Laplacian $\Delta_{HR}$ acting on coexact 1-forms (with absolute condition when $\bd\Omega$ is not empty). This lead us to collect in sections \ref{consequences} and \ref{second} a lot of already known results where we have an explicit control of $\lambda_{1,1}''(M,g)$. 

\smallskip
A last point, is that, when $A$ is closed  (i.e. zero magnetic field), we can establish a \emph{sharp} upper bound for the first eigenvalue $\lambda_1(H_{A,q})$ in term of the distance of $A$ to the an integral  lattice of harmonic $1$-forms (see Theorem \ref{harmonic potential}) and we can discuss the equality case (see Theorems \ref{harmonic potential} and \ref{genusone}). 

\smallskip
In the rest of the introduction we will recall some known facts and discuss the main results.

%%%%%

\subsection{Preliminary facts and notation} \label{preliminary} First, we recall  the absolute boundary conditions for a differential $p$-form $\omega$. A  form $\omega$ is said to be {\it tangential} if $i_NA=0$ on $\bd M$, where $N$ denote the exterior normal vector to the boundary;  then,  $\omega$ satisfies the \emph{absolute boundary conditions} if $\omega$ and $d\omega$ are both tangential. We denote $\lambda_{1,p}$ the first eigenvalue of the Hodge Laplacian on $p$-forms (with absolute boundary conditions if $\partial M$ non empty), and by $\lambda_{1,p}''$ (resp. $\lambda_{1,p}'$) the first eigenvalue when restricted to co-exact (resp. exact) $p$-forms. It follows that
$$
\lambda_{1,p}\leq \min\{\lambda_{1,p}',\lambda_{1,p}''\}
$$
and as $\lambda_{1,p}''=\lambda_{1,p+1}'$ (by differentiating eigenfunctions) we see 
$$
\lambda_{1,p}''\geq \max\{\lambda_{1,p},\lambda_{1,p+1}\}.
$$
In particular,
$$
\lambda_{1,1}''\geq \max\{\lambda_{1,1},\lambda_{1,2}\}.
$$

We recall now the variational definition of the spectrum. Let $M$ be a compact manifold. If the boundary is non empty, we assume for $u\in C^{\infty}(M,\bf C)$ the magnetic Neumann conditions, as in \eqref{mneumann}. Then one verifies that
$$
\int_{M}(H_{A,q}u)\bar uv_g=\int_{M}(\abs{\nabla^Au}^2+q\abs{u}^2)v_g,
$$
and the associated quadratic form is then
$$
Q_{A,q}(u)=\int_{M}(\abs{\nabla^Au}^2+q\abs{u}^2)v_g.
$$

We also introduce the Rayleigh quotient of a smooth function $u\not =0$, defined by
 \begin{equation}
R_{A,q}(u)= \frac{Q_{A,q}(u)}{\Vert u\Vert^2}
\end{equation}

The spectrum of $H_{A,q}$ admits the usual variational characterization:

\begin{equation}
\lambda_1(H_{A,q})= \min\Big\{ R_{A,q}(u)\ u\in C^{1}(M,\bf C) / \{0\}\Big\}
\end{equation}
and
\begin{equation}
\lambda_k(H_{A,q})= \min_{E_k} \max \Big\{ R_{A,q}(u):\ u\in E_k  / \{0\}\Big\}
\end{equation}
where $E_k$ runs through the set of all $k$-dimensional vector subspaces of $C^{1}(M,\bf C)$.

\medskip
The following proposition recalls some well-known facts.   If $c$ is a closed curve (a loop), the quantity
\begin{equation} \label{flux}
\Phi^A_c=\dfrac{1}{2\pi}\oint_{c}A
\end{equation}
is called the {\it flux} of $A$ across $c$. We will not  specify the orientation of the loop, so that the flux will only be defined up to sign. This will not affect any of the statements, definitions or results which we will prove in this paper.

\begin{prop}  \label{basic facts}
\begin{enumerate}
\item
The spectrum of $H_{A,q}$ is equal to the spectrum of $H_{A+d\phi,q}$ for all smooth real valued functions $\phi$; in particular, when $A$ is exact, the spectrum of $H_{A,q}$ reduces to that of the classical Schr\"{o}dinger operator with potential $q$ acting on functions (with Neumann boundary conditions if $\bd M$ is not empty).

\item 
 Let $A$ be $1-$form on $M$. Then, there exists a smooth real valued function $\phi$ on $M$ such that the $1$-form $\tilde A=A+d\phi$ is co-closed and tangential, that is:
\begin{equation} \label{boundary conditions}
\delta\tilde A=0, \,\, i_N\tilde A=0.
\end{equation}

\item 
 Set:
$$
{\rm Har}_1(M)=\Big\{h\in\Lambda^1(M): dh=\delta h=0 \,\,\text{on $M$},\,\, i_Nh=0\,\text{on}\,\, \partial M\Big\}.
$$
Assume that the $1$-form $A$ is co-closed and tangential. Then $A$ can be decomposed
\begin{equation} \label{decomposition}
A=\delta\psi+h,
\end{equation}
where $\psi$ is a smooth tangential $2$-form and $h\in{\rm Har}_1(M)$. Note that the vector space ${\rm Har}_1(M)$ is isomorphic to the first de Rham absolute cohomology space $H^1(M,\reals)$. 

\item
 We have $\lambda_1(H_{A,0})\doteq \lambda_1(\Delta_A)=0$ if and only $A$ is closed (i.e. $B=0$) and  the cohomology class of $A$ is an integer (that is $\Phi^A_c\in \bf Z $  for any loop $ c$ in $M$).

\end{enumerate}
\end{prop}

{\rm Assertion (1) expresses the well-known {\it Gauge invariance} of the spectrum. Thanks to Assertion (2), in the study of the spectrum of the magnetic Laplacian, we can always assume that the potential $A$ is co-closed and tangential.}  

\begin{proof} 
\begin{enumerate}

\item 
 This comes from the fact that
\begin{equation}\label{gauge}
\Delta_A e^{-i\phi}=e^{-i\phi} \Delta_{A+d\phi}
\end{equation}
hence $\Delta_A$ and $\Delta_{A+d\phi}$ are unitarily equivalent.

\item 
Observe that the problem:
$$
\twosystem
{\Delta\phi=-\delta A \quad\text{on}\quad M,}
{\frac{\partial\phi}{\partial N}=-A(N)\quad\text{on}\quad\partial M}
$$
has a unique solution (modulo a multiplicative constant). It is immediate to verify that $\tilde A=A+d\phi$ is indeed co-closed and tangential.

\item
 We apply the Hodge decomposition to the $1$-form $A$ (see \cite{Sc}, Thm. 2.4.2), and get:
\begin{equation}
A=df + \delta \psi +h,
\end{equation}
where $f$ is a function which is zero on the boundary, $\psi$ is a tangential 2-form and $h$ is a 1-form satisfying $dh=\delta h=0$ (in particular, $h$ is harmonic). Now, as $\delta A=0$ we obtain $\delta df=0$ hence $f$ is a harmonic function; since $f$ is zero on the boundary, we get $f=0$ also on $M$ and we can write
\begin{equation}\label{coulomb}
A=\delta \psi +h.
\end{equation}
Now, since both $A$ and $\delta\psi$ are tangential, also $h$ will be tangential.

\item
This result was proved by Shigekawa \cite{Sh} for closed manifolds; for Neumann boundary conditions  see also \cite{HHHO}. 
\end{enumerate}
\end{proof}

\nero {\it In the sequel, when we write the decomposition $A=\delta \psi +h$, it will be implicitely supposed that $\psi$ is a tangential 2-form and $h$ is a 1-form satisfying $dh=\delta h=0$ and $i_Nh=0$.} 

\medskip

From definition \eqref{connection} we see
\begin{equation}\label{magenergy}
\vert \nabla^Au\vert^2 =\vert du\vert^2 + \vert A\vert^2 \vert u\vert^2 +2\mbox{Im}\scal{A}{\bar u du}.
\end{equation}

Since $A$ is real, it is clear that if $u\in C^\infty(M,\reals)$ is a real valued function, then $\mbox{Im}\scal{A}{\bar u du}=0$ and, then,
$$R_{A,q}(u)=\frac{\int_M\left(\vert du\vert^2 +( \vert A\vert^2 +q)\vert u\vert ^2\right) v_g}{\int_M\vert u\vert ^2 v_g}.$$
Since  $C^2(M,\reals)$ is a subspace of  $C^2(M,\bf C)$, it follows that  the eigenvalues of $H_{A,q} $ are dominated by those of the scalar Schr\"{o}dinger operator $H_{0,\vert A\vert^2+q} = \Delta +\vert A\vert^2+q $, that is
\begin{equation}\label{comparaison}
\lambda_k( H_{A,q} )\le \lambda_k( H_{0,\vert A\vert^2+q }) =  \lambda_k( \Delta + \vert A\vert^2+q ).
 \end{equation}   
 
     For the first eigenvalue of $H_{A,q} $, one also has a lower estimate by the first eigenvalue of the scalar Schr\"{o}dinger operator $H_{0,q} = \Delta+q$; in other words: 
     
\begin{equation}\label{lavineocarroll1}
\lambda_1( H_{A,q} )\ge \lambda_1( H_{0,q }).
\end{equation}

This property can be seen as an immediate consequence of the so-called diamagnetic inequality (see for instance Theorem 2.1.1 in \cite{FH1}). 

%In what follows we provide an alternative proof based on Lavine-O'Carroll identity which allows to characterize the equality case. This proof is an immediate extension to the context of Riemannian manifolds of arguments used by Helffer  in the case of bounded domains of $\reals^n$ under Dirichlet boundary conditions (see Proposition 1.1 in \cite{He}). %Its extension to the context of compact Riemannian manifolds does not require new ideas. For reader's convenience, we provide a proof that follows arguments of \cite{Helffer}, . 

 \subsection{Statement of results}
 
Before stating the results, let us define a distance associated to the 1-form $A$ which will play an important role in our estimates (see \eqref{ldistance} below). Let 
${\cal L}_{\bf Z}$ be the lattice in ${\rm Har}_1(M)\sim H^1(M,\reals)$ formed by the integral harmonic $1$-forms (those having integral flux around any loop). 
Given $A\in {\rm Har}_1(M)$, we define its distance to the lattice ${\cal L}_{\bf Z}$ by the formula:
\begin{equation}\label{ldistance}
d(A,{\cal L}_{\bf Z})^2=\min \Big\{\norm{\omega-A}^2,\, \omega \in {\cal L}_{\bf Z}\Big\},
\end{equation}
where $\norm{\cdot}$ denotes the $L^2$-norm of forms in $M$. Of course, when $H^1(M,\reals)=0$ any harmonic $1$-forms is zero and we set $d(A,{\cal L}_{\bf Z})=0$.

 \begin{thm} \label{firsteigenvalue} Let $H_{A,q}$ be a magnetic Schr\"{o}dinger operator on a compact Riemannian manifold $(M,g)$ of dimension $n$, where $A=\delta\psi+h$ is a potential as in (\ref{decomposition}). One has, under Neumann boundary conditions if $\partial M \not = \emptyset$:

 \begin{enumerate}
\item
 \begin{equation} \label{generalestimate}
\l_1(H_{A,q})\leq \Gamma(M,A,q):=\frac{1}{\vert M\vert}\left( d(h,{\cal L}_{\bf Z})^2 +\frac{\Vert B\Vert^2}{\lambda_{1,1}''(M) } + \int_{M} q v_g \right)
\end{equation}
where $\vert M\vert$ denotes the volume of $M$ and $\lambda_{1,1}''(M)$ isthe first eigenvalue of the Hodge-De Rham Laplacian $\Delta_{HR}$ acting on co-exact $1$-forms (with absolute boundary condition if $\partial M \not=\emptyset$).

\item
If the first absolute De Rham cohomology group vanishes : $H^1(M,\reals)=0$, then 
\begin{equation} \label{nocohomology}
\l_1(H_{A,q})\leq \frac{1}{\vert M\vert}\left(\frac{\Vert B\Vert^2}{\lambda_{1,1}''(M) } + \int_{M} q v_g\right)
\end{equation}
with equality if and only if $\Delta_{HR} (\delta \psi)=\lambda_{1,1}''(\delta \psi)$ and $\vert \delta \psi\vert^2+q$ is constant, equal to $\lambda_1(H_{A,q})$. 

\end{enumerate}
 \end{thm}

The case when the potential $A$ is closed (that is $B=0$) is of special interest. We have

\begin{thm}\label{harmonic potential}
Let $H_{A,q}$ be a magnetic Schr\"{o}dinger operator on a compact Riemannian manifold $(M,g)$ of dimension $n$, where the potential $A$ is closed, so that we can write $A=h$ as in (\ref{decomposition}). One has under Neumann boundary conditions if $\partial M \not = \emptyset$:

\begin{equation} \label{specificestimate}
\l_1(H_{A,q})\leq\frac{ d(h,{\cal L}_{\bf Z})^2+\int_{M}qv_g}{\vert M\vert}.
\end{equation}

In case of equality in (\ref{specificestimate}), there exists an integer harmonic form $\omega \in {\cal L}_{\bf Z}$ such that 
$\vert A-\omega\vert^2+q$ is constant. In particular, if the potential $q$ is constant, $(M,g)$ carries a harmonic $1$-form of constant length.

\end{thm}

Sometimes we can characterize equality. Precisely:

\begin{thm}\label{genusone}
\begin{enumerate}
\item
When $(M,g)$ is a flat torus, we have equality in (\ref{specificestimate}) if and only if the potential $q$ is constant.
\item
When  $M$ is a two-dimensional torus (that is, a genus one surface) and $q$ is constant
with have equality in (\ref{specificestimate})  if and only if $(M,g)$ is a flat torus.
\end{enumerate}
\end{thm}

In section \ref{first} we will give applications of Theorem \ref{firsteigenvalue} for manifolds  for which we have a good control of $\lambda_{1,1}''(M,g)$. First of all, using the Bochner formula,  we show such control for closed manifolds with Ricci curvature bounded below by a positive constant; when the boundary is not empty, we have to impose that it is convex. Then, we extend such lower bound also when the inner curvature is not everywhere positive; for example, for convex domains in $\real n$, and for hypersurfaces of manifolds with curvature operator with arbitrary sign, provided that the extrinsic curvatures are large enough. The general principle is that one still has a positive lower bound for $\lambda_{1,1}''$ if the positivity of the principal curvatures of the boundary compensate, in some sense, for the negativity of the inner curvature. \rm

\medskip
Thanks to the Li-Yau \emph{conformal volume} $V_c(M)$ associated to the Riemannian manifold $(M,g)$,  which depends only on the conformal class of $g$ (see \cite{EI1} for a definition
and detail about it), and the results in \cite{EI2}, it is also possible to get an upper bound for the second eigenvalue.

\begin{thm} \label{volconf} Let $H_{A,q}$ be a magnetic Schr\"{o}dinger operator on a compact Riemannian manifold $(M,g)$ of dimension $n$, where $A=\delta\psi+h$ is a potential as in (\ref{decomposition}). One has (under Neumann boundary conditions if $\partial M \not = \emptyset$) :

\begin{equation} \label{confvol_ineq}
\lambda_2(H_{A,q})\le n\frac{V_c(M)}{\vert M\vert}+\Gamma (M,A,q)
\end{equation}
with $\Gamma(M,A,q)$ as in (\ref{generalestimate}).
\end{thm}

In section \ref{second}, we will give applications of Theorem \ref{volconf} in specific situations.

\medskip
We then state an upper bound valid for all the eigenvalues.

\begin{thm} \label{main1} Let $H_{A,q}$ be a magnetic Schr\"{o}dinger operator on a closed Riemannian manifold $(M,g)$ of dimension $n$, where $A=\delta\psi+h$ is a potential as in (\ref{decomposition}).  
\begin{enumerate}
\item
There exists a constant $c([g])$ depending on the conformal class of $g$ such that

\begin{equation} \label{conformal}
\lambda_k (H_{A,q})\le \Gamma(M,A,q)+c([g])\left(\frac{k}{\vert M\vert}\right)^{2/n}.
\end{equation}

\item
If $(M^n,g)$ has a Ricci curvature bound
$
{\rm Ric}(M,g)\ge -a^2(n-1)
$ 
and if $\vert A\vert^2+q \ge 0$ (in particular, if $q\ge 0$),
there exist positive constants $c_1,c_2,c_3$ depending only on the dimension $n$ of $M$ such that
\begin{equation} \label{Ricci}
\lambda_k (H_{A,q}))\le c_1\Gamma(M,A,q)+c_2a^2+c_3\left(\frac{k}{\vert M\vert}\right)^{2/n},
\end{equation}
with $\Gamma(M,A,q)$ as in (\ref{generalestimate}).

\end{enumerate}
\end{thm}

These result will be deduced from inequality (\ref{comparaison}) and from estimates for the Schr\"{o}dinger Laplacian derived in \cite{Ha} and \cite{GNS}. Note that if $A=0$, we recover the result of \cite{CM} for the usual Laplacian.

\medskip
In the specific situation of an Euclidean domain, we get other estimates in Theorem \ref{lambdak} using Riesz means, as a corollary of Inequality (\ref{comparaison}) and of \cite{EHIS}.

%\section{The first eigenvalue of $H_{A,q}$.} \label{first}

\section{Upper bounds for the first eigenvalue of $H_{A,q}$} \label{first}

\subsection{Proof of Theorem \ref{firsteigenvalue}.}
 We recall that $A=\delta\psi+h$ denotes the potential, $\psi$ is a smooth tangential $2$-form, $h\in{\rm Har}_1(\Omega)$,  $\lambda_{1,1}''(\Omega)$ denotes the first eigenvalue of the Laplacian acting on co-exact $1$-forms,  $B=dA$ is the curvature of the potential $A$, and ${\cal L}_{\bf Z}$ denotes the integral lattice of $H^1(M)$ formed by the integer harmonic $1$-forms ${\rm Har}_1(M)$.

 \medskip
Let  $ \omega \in {\cal L}_{\bf Z}$.  
  Fix a base point $x_0$ and define, for $x\in\Omega$:
\begin{equation}\label{path}
\phi(x)\doteq\int_{x_0}^x\omega,
\end{equation}
where on the right we mean integration of $\omega$ along any path joining $x_0$ with $x$. As $\omega$ is closed, $\phi(x)$ does not depend on the choice of two homotopic paths and since the flux of $A$ across each $c_j$ is an integer, $\phi(x)$ is multivalued and defined up to $2\pi \bf Z$. This implies that the function $u(x)=e^{i\phi(x)}$ is well defined. 
As 
$
d\phi=\omega
$
we see that
$
du=iu\omega
$
and therefore
$$
\nabla^Au=du-iuh-iu\delta \psi=iu(\omega-h-\delta \psi).
$$
Since $\abs{u}=1$, we obtain:
$$
\abs{\nabla^Au}^2=\abs{\omega-h-\delta \psi}^2.
$$
We use $u(x)$ as test-function for the first eigenvalue of $\Delta_A$. Then, for each $\omega\in {\cal L}_{\bf Z}$, we have the relation

\begin{equation} 
\lambda_1(H_{A,q})\leq \dfrac{\int_{M}\abs{\nabla^Au}^2v_g+\int_M\vert u\vert^2 qv_g}{\int_{M}\abs{u}^2v_g}=
\dfrac{\norm{\omega-h-\delta \psi}^2}{\vert M\vert}+\frac{\int_M qv_g}{\vert M\vert}
 \end{equation}  

As $\omega-h$ is harmonic, it is $L^2$-orthogonal to $\delta \psi$ and we get

\begin{equation}\label{ltwo}
\lambda_1(H_{A,q})\leq \dfrac{\norm{\omega-h}^2+\norm{\delta \psi }^2+\int_M qv_g} {\vert M\vert}
\end{equation}
Now observe that, since $\delta\psi$ is coexact and tangential, one has by the variational characterization of the eigenvalue $\lambda_{1,1}''(M,g)$:
$$
\frac{\int_{M} \vert d \delta \psi\vert^2v_g}{\int_{M} \vert \delta \psi\vert^2v_g} \ge \lambda_{1,1}''(M,g)
$$
 As $d \delta\psi=B$, we have
\begin{equation}\label{lambdaoneone}
\int_{M} \vert \delta \psi \vert^2v_g \le \frac{1}{\lambda_{1,1}''(M,g)}\Vert B\Vert^2.
\end{equation}
Taking the infimum on the right-hand side of \eqref{ltwo} over all $\omega\in{\cal L}_{\bf Z}$ we obtain, taking into account \eqref{lambdaoneone}:
$$
\l_1(H_{A,q})\leq\frac{ d(h,{\cal L}_{\bf Z})^2}{\vert M\vert}+\frac{\Vert B\Vert^2}{\lambda_{1,1}''(M)\vert M\vert}+\frac{1}{\vert M\vert}\int_M qv_g
$$
as asserted.

\medskip
When $H^1(M,\reals)=0$, we have immediately the relation
$$
\l_1(H_{A,q})\leq\frac{1}{\vert M\vert}\left(\frac{\Vert B\Vert^2}{\lambda_{1,1}''(M,g) } + \int_{M} q v_g\right).
$$
In case of equality, we must have equality in all the step of the proof: in particular, we must have
$$
\frac{\int_M\vert d\delta\psi\vert^2 v_g}{\int_M\vert \delta \psi \vert^2v_g} =\lambda_{1,1}''(M,g)
$$
which means that $\lambda_{1,1}''$ is an eigenvalue for the eigenfunction $\delta \psi$. For $u=1$, Equation (\ref{maglap})  becomes
$$
\Delta_A u=\vert\delta \psi\vert^2 
$$
and the equation $H_{A,q}u=\lambda_1(H_{A,q})u$ becomes
$$
\vert\delta \psi\vert^2 +q= \lambda_1(H_{A,q}).
$$
as asserted.

\subsection{Proof of Theorem \ref{harmonic potential}}

\begin{enumerate}
\item
Inequality (\ref{specificestimate}) is an immediate consequence of Inequality (\ref{generalestimate}).

\item
In order to investigate the equality case, we will derive the inequality using a different approach. Let $\omega \in {\cal L}_{\bf Z}$ and $u=e^{i\phi}$ the associated function on $M$ as defined in  (\ref{path}). Recall that $\abs{\cdot}$ denotes the pointwise norm, thus defining a smooth function on $M$. 

\smallskip
First we observe that
\begin{equation} \label{equality}
\Delta_A u = \vert A-\omega \vert^2 u.
\end{equation}

In  fact recall that, as $\delta A=0$:
$$
\Delta_Au=\Delta u+\abs{A}^2u+2i\scal{du}{A}.
$$
As $du=iu\omega$ one gets:
$$
\Delta u=\delta du=\delta(iu\omega)=i(-\scal{du}{\omega}+u\delta\omega)=-i\scal{du}{\omega}=\abs{\omega}^2u
$$
and \eqref{equality} follows after an easy computation. In turn, one has:
$$
H_{A,q}u=\abs{A-\omega}^2u+qu.
$$
Using $u$ as a test-function, and recalling that$\vert u\vert^2=1$, we have
\begin{equation} \label{eigenvalue}
\lambda_1(H_{A,q})\int_{M}\abs{u}^2v_g\leq \int_M\scal{H_{A,q}u}{u}v_g = \Vert \omega- A\Vert^2 + \int_M qv_g.
\end{equation}

In particular, if we choose $\omega$ so that $d(\omega,A)^2= d(A,{\cal L}_{\bf Z})^2$ we recover inequality (\ref{specificestimate}). But now, if equality holds, we  see that $u$ must be an eigenfunction for $\lambda_1(H_{A,q})$, that is
$$
\lambda_1(H_{A,q}) u=H_{A,q}u=\Delta_Au+qu=(\vert A-\omega\vert^2 +q)u.
$$

So, we deduce that $\vert A-\omega\vert^2 +q=\lambda_1(H_{A,q})$ as asserted. In particular, if $q$ is constant, $\vert A-\omega\vert$ is constant, and $(M,g)$ carries a harmonic $1$-form of constant length.
\end{enumerate}

%\item
%In the case of a flat torus, we are able to compute the spectrum of $\Delta_A=H_{A,0}$ (if $A$ is harmonic) and this allows to show the sharpness and to study the situation of $H_{A,q}$. 

%

\subsection{Spectrum of flat tori}
 
In order to prove Theorem \ref{genusone}, we investigate the spectrum of flat tori. Let $\Sigma$ be a flat n-dimensional torus, quotient of $\real n$ by a lattice $\Gamma$. Recall that the dual lattice $\Gamma^{\star}$ is defined by 
$$
\Gamma^{\star}=\{v: \scal{v}{w}\in{\bf Z} \quad\text{for all} \quad w\in\Gamma\}
$$
On a flat torus any harmonic $1$-form $\xi$ is parallel, and then it has constant pointwise norm $\abs{\xi}$. In particular
$$
\norm{\xi}=\abs{M}\abs{\xi}.
$$
The lattice ${\cal L}_{\bf Z}$ is an additive subgroup of the vector space of harmonic (hence parallel) $1$-forms. If $\omega$ is one such consider the associated dual parallel vector field, $\omega^{\sharp}$. We remark that this induces an isomorphism of groups:
\begin{equation}\label{iso}
{\cal L}_{\bf Z}\cong 2\pi\Gamma^{\star}.
\end{equation}
To prove that, associate to each $X\in\Gamma$  the curve $c_{X}:[0,1]\to\Sigma$ given by $c_X(t)=tX$. Note that $c_X$ is a loop because $\Sigma$ is $\Gamma-$invariant. The flux of $\omega$ across $c_X$ is easily seen to be
$$
\Phi^{\omega}_{c_X}=\dfrac{1}{2\pi}\omega(X)=\dfrac{1}{2\pi}\scal{\omega^{\sharp}}{X}.
$$
Hence any such flux is an integer if and only if $\scal{\omega^{\sharp}}{X}\in 2\pi\bf Z$. This is true for all $X\in\Gamma$ iff $\omega^{\sharp}\in 2\pi\Gamma^{\star}$, which proves \eqref{iso}. 

\smallskip

Now if $\omega\in{\cal L}_{\bf Z}$, it is readily seen that the associated function $u$
as in (\ref{path}) is given by:
$$
u(x)=e^{i\scal{\omega^{\sharp}}{\,x}}
$$
which is well-defined on $(M,g)=\real n/\Gamma$.  
Hence, for each $\omega\in{\cal L}_{\bf Z}$, thanks to (\ref{equality}), we have:
$$
\Delta_Au=\abs{A-\omega}^2u
$$
and the constant $\abs{A-\omega}^2$ is thus an eigenvalue of $\Delta_A$ associated to the eigenfunction $u$.
Because of \eqref{iso} the set
$$
\{u(x)=e^{i\scal{\omega^{\sharp}}{\,\,x}}, \omega\in {\cal L}_{\bf Z}\}
$$
gives rise to a complete orthonormal basis of $L^2(M)$, hence we have found all the eigenvalues of $\Delta_A$.  In conclusion, we have the following fact.

\begin{prop}\label{flattorus} Let $\Sigma$ be a flat torus, quotient of $\real n$ by the lattice $\Gamma$ and let $\Gamma^{\star}$ denote the lattice dual to $\Gamma$. Let $A$ be a harmonic $1$-form. Then the spectrum of the magnetic Laplacian with potential $A$, that is, the operator $\Delta_A=H_{A,0}$, is given by 
$$
\{\abs{A-\omega}^2: \omega\in {\cal L}_{\bf Z}\cong 2\pi\Gamma^{\star}\}
$$
with associated eigenfunctions
$
\{u(x)=e^{i\scal{\omega^{\sharp}}{\,\,x}}\}.
$
In particular
$$
\lambda_1(\Delta_A)=\inf_{\omega\in{\cal L}_{\bf Z}}\abs{A-\omega}^2.
$$
 \end{prop}

 \subsection{Proof of Theorem \ref{genusone}}

1. Now let $M$ be a flat torus, $A=h$  a harmonic $1$-form, and let $\omega_0$  be an element in ${\cal L}_{\bf Z}$ such that 
$$
d(A,{\cal L}_{\bf Z})^2=\norm{A-\omega_0}^2=\abs{M}\abs{A-\omega_0}^2.
$$
Inequality \eqref{specificestimate} takes the form:
$$
\lambda_1(H_{A,q})\leq \abs{A-\omega_0}^2+\dfrac{1}{\abs M}\int_M qv_g
$$
with equality if and only if the associated test-function $u(x)=e^{i\scal{\omega_0^{\sharp}}{\, x}}$ (of constant modulus one) is an eigenfunction of $H_{A,q}$. As
$$
H_{A,q} u=\Delta_Au+qu=(\abs{A-\omega_0}^2+q)u
$$
we see indeed that we have equality in \eqref{specificestimate} if and only if
$$
q=\dfrac{1}{\abs{M}}\int_M q v_g
$$
that is, iff $q$ is constant. 

\smallskip
2. Now assume that $(M,g)$ is a genus one surface and $q$ is constant. It remains to show that, if equality holds, $M$ has to be flat. 

Since $q$ is constant, there exists a harmonic one form $\xi$ with constant length by the second assertion of Theorem \ref{harmonic potential}. We will apply Bochner formula to $\xi$.
Let $\Delta_{HR}$ the Laplacian on 1-forms and $\nabla$ the covariant derivative. Bochner's identity gives, for  any $1$-form $\alpha$:
\begin{equation} \label{Bochner}
\scal{\Delta_{HR}\alpha}{\alpha}= \vert \nabla\alpha\vert^2 + \frac 12 \Delta\vert\alpha\vert^2 +\mbox{Ric} (\alpha,\alpha).
\end{equation}

In dimension $2$ one has ${\rm Ric}=Kg$, where $K$ is the Gaussian curvature.  As $\xi$ is harmonic and of constant pointwise norm, we get 
$$
0=\int_M\abs{\nabla\xi}^2v_g+\abs{\xi}^2\int_MKv_g.
$$
As $M$ has genus one we see $\int_M Kv_g=0$; this means that $\xi$ must actually be parallel. But then $\star\xi$ must also be parallel ; by normalization, we have a global orthonormal basis $(\xi,\star\xi)$ of parallel one forms, which forces $(M,g)$ to be flat. 

%%%

\subsection{A few consequences} \label{consequences}

We can now describe a few consequences of Theorem \ref{firsteigenvalue} in some specific situations where we are able to control the eigenvalue $\lambda_{1,1}''$ of the manifold $M$.

\subsubsection{Positive Ricci curvature} 

 When the Ricci curvature of $M$ is positive (and $\partial M$ is convex if nonempty), then  $H^1(M,\reals)=\{0\}$. This implies that the harmonic part $h$ in the decomposition \eqref{coulomb} of the potential $1$-form $A$ vanishes, so that $A=\delta\psi$ for a tangential two-form $\psi$.  Moreover, the constant $\lambda_{1,1}''(M)$ can be controlled in terms of a lower bound of  the Ricci curvature of $M$. Indeed, we have the  
 
\begin{lemme} \label{lem ric}
  Let $(M,g)$ be a compact Riemannian manifold  whose Ricci curvature satisfies 
$$ \mbox{Ric}\ge  c\  g$$
for some positive $c$.  When $\partial M\ne \emptyset$,  assume furthermore that $\partial M$ is convex (i.e. its shape operator $S$ is nonnegative). One has
$$\lambda_{1,1}''(M)\ge 2c$$
Moreover, the equality holds if and only if every co-exact eigenform $\alpha$ associated with $\lambda_{1,1}'' (M)$ is such that $\alpha^\sharp$ is a Killing vector field which satisfies  $Ric(\alpha^\sharp) = c\alpha^\sharp$ and, when $\partial M\ne \emptyset$, $S(\alpha^\sharp)=0$.  
%and $M$ admits  at least one nonzero Killing vectorfield (satisfying the absoulte boundary conditions if $\partial M\ne\emptyset$), then  $\lambda_{1,1}''(M)=2c$.
 \end{lemme}
 
 Here $S:T\bd M\to T\bd M$ is the shape operator of $\bd M$, defined as follows: if $N$ is the inner unit normal vector to the boundary, and $X\in T\bd M$,  then $S(X)=-\nabla_XN$. 

\begin{proof}
%The positivity of the Ricci curvature of $M$ implies that $b_1(M)=0$. Thus, according to Corollary \ref{cor b1=0}, it suffices to prove that $\lambda_{1,1}''(M)$ is  bounded below by $2c$. 

We use again the Bochner identity (\ref{Bochner}):
$$\scal{\Delta_{HR}\alpha}{\alpha}= \vert \nabla\alpha\vert^2 + \frac 12 \Delta\vert\alpha\vert^2 +\mbox{Ric} (\alpha,\alpha).$$
On the other hand, we have the following general  inequality (see \cite{GM}, Lemma 6.8 p. 270)
\begin{equation}\label{Gallot-Meyer}
\vert \nabla\alpha\vert^2\ge \frac 12 \vert d\alpha\vert^2 +\frac 1n \vert \delta\alpha\vert^2\ge \frac 12 \vert d\alpha\vert^2 
\end{equation}
in which the equality holds if and only if $\nabla\alpha$ is anti-symmetric, that is $\alpha^\sharp$ is a Killing vectorfield (see \cite{Be} Theorem 1.81, p. 40).
 When  $\alpha$ is a co-exact eigenform associated to $\lambda_{1,1}''(M)$, we have
\begin{equation}\label{bochner}
\lambda_{1,1}''(M)\vert \alpha\vert^2 = \scal{\Delta_{HR}\alpha}{\alpha}\ge \frac12 \vert d\alpha\vert^2+ \frac 12 \Delta\vert\alpha\vert^2 +c \vert \alpha\vert^2
\end{equation}
When $M$ is closed, one has $\int_M\Delta\vert\alpha\vert^2 v_g=0$ and $\int_M \vert d\alpha\vert^2 v_g=\lambda_{1,1}''(M)\int_M\vert \alpha\vert^2 v_g
$ and the result follows from \eqref{bochner} after integration, that is 
\begin{equation}\label{bochnerinteg}
\lambda_{1,1}''(M)\int_M\vert \alpha\vert^2v_g \ge \frac12 \int_M\vert d\alpha\vert^2v_g+ c\int_M \vert \alpha\vert^2 v_g = \left(\frac{\lambda_{1,1}''(M)}2 +c \right) \int_M\vert \alpha\vert^2v_g
\end{equation}
which implies $\lambda_{1,1}''(M)\ge 2c$. 

If $\partial M\ne \emptyset$, we   observe that, since $i_N\alpha=0$ and $i_N d\alpha=0$, the vector field $\alpha^\sharp$ is tangent along the boundary and $0=d\alpha(N,\alpha^\sharp) =  \nabla_N\alpha({\alpha^\sharp}) -\nabla_{\alpha^\sharp}\alpha(N) =\frac 12 N\cdot   \vert \alpha\vert^2 - \scal{S(\alpha^\sharp)}{\alpha ^\sharp} $. Thus Green formula gives
$$\int_M\Delta\vert\alpha\vert^2 v_g=\int_{\partial M} N\cdot \vert \alpha\vert^2  v_g =2\int_{\partial M} \scal{S(\alpha^\sharp)}{\alpha ^\sharp} v_g\ge0$$
The rest of the proof is the same as above.

\smallskip
 
Let us discuss the equality case when $\partial M $ is empty. Assume that a co-exact eigenform $\alpha $ satisfies : $\alpha^\sharp$ is a Killing vector field and $Ric(\alpha^\sharp,\alpha^\sharp) = c\vert\alpha\vert^2$. Under these conditions,  the equality holds in the inequality \eqref{bochner} and, then, in  \eqref{bochnerinteg}, which implies $\lambda_{1,1}''(M)=2c$.

Conversely,  if $\lambda_{1,1}''$(M)$ = 2c$, then, for any co-exact eigenform $\alpha$,  the equality holds in \eqref{bochnerinteg} and, then, in \eqref{bochner} and \eqref{Gallot-Meyer} which implies that 
$Ric(\alpha^\sharp,\alpha^\sharp) = c\vert\alpha\vert^2$ and that $D\alpha$ is anti-symmetric,  that is $\alpha^\sharp$ is a Killing vector field.

When $\partial M $ is not empty, the discussion of the equality case follows the same lines  observing that since $\partial M $ is convex, the equality $\scal{S(\alpha^\sharp)}{\alpha ^\sharp}=0$ occurs if and only if $S(\alpha^\sharp)=0$.
\end{proof}
An immediate consequence of Theorem \ref{firsteigenvalue} is the
%%%%%%%%%
 \begin {cor} \label{cor ric}
 Under the circumstances   of Theorem  \ref{firsteigenvalue} and the  assumption that the Ricci curvature of $M$ satisfies 
 $ \mbox{Ric}\ge  c\  g$
for some positive $c$, and that the boundary $\partial M$ is convex (if nonempty), one has 
\begin{equation}\label{ineqric}
\lambda_1( H_{A,q} )\le \frac 1{\vert M\vert}\left(\frac {\Vert B\Vert^2}{2c}  +\int_Mq v_g\right)
\end{equation}   
where $B=dA$ is the magnetic field. The equality holds in \eqref{ineqric} if and only if  
$A^{\sharp}$ is a Killing vector field with $Ric (A^{\sharp})=c A^{\sharp}$, $\vert A^{\sharp}\vert ^2 +q = \lambda_1( H_{A,q} )$ and, when $\partial M\ne \emptyset$, $S(A^{\sharp})=0$.
 \end{cor}

Recall that Bochner vanishing Theorem tells us that a non Ricci-flat  manifold $M$ with non-negative Ricci curvature and mean-convex boundary if $\partial M\ne\emptyset$, satisfies $H^1(M,\reals)=\{0\}$. On the other hand, Bochner's identity gives for any Killing vector field $A$, $\Delta_{HR}A=2 \,{\rm Ric}(A)$.

\smallskip
The inequality \eqref{ineqric} improves by a factor 2  the estimate obtained  by Cruzeiro, Malliavin and Taniguchi (\cite{CMT}, Theorem 1.1) for $\lambda_1( H_{A,0} )$ for closed  manifolds (be careful, the magnetic Laplacian defined in \cite{CMT} coincides with $ \frac 12 \Delta_{A}$).

\smallskip
An important special case we want to emphasize is the following

 \begin {cor} \label{cor sphere}
(i) Let $H_{A,q} $  be a magnetic Schr\"{o}dinger operator   on the standard $n$-dimensional sphere $\mathbb S^n$. One has 
\begin{equation}\label{ineqsphere}
\lambda_1( H_{A,q} )\le \frac 1{\sigma_n} \left(\frac {\Vert B\Vert^2}{2(n-1)}  +\int_{\mathbb S^n} q v_g\right)
\end{equation}   
where $\sigma_n=(n+1)\omega_{n+1}$ is the volume of   $\mathbb S^n$ and $B=dA$ is the magnetic field. The equality holds in \eqref{ineqsphere} if and only if $A^{\sharp}$ is a Killing vector field of $\mathbb S^n$ and $\vert A^{\sharp}\vert ^2 +q= \lambda_1( H_{A,q} )$.

\smallskip

(ii) Let $H_{A,q} $  be a magnetic Schr\"{o}dinger operator   on a spherical cap $C_r(x_0)$ of radius $r\le \frac \pi 2$ centered at $x_0$.  One has
\begin{equation}\label{sphericalcap}
\lambda_1( H_{A,q} )\le \frac 1{v_n ( r )} \left(\frac {\Vert B\Vert^2}{2(n-1)}+\int_{C_r}q v_g\right) 
\end{equation}
where  $v_n( r ) = \sigma_{n-1}\int_0^r (\sin t)^{n-1} dt$ is the volume of $C_r(x_0)$.

 %\begin {proof} The inequality \eqref{ineqsphere} follows from \eqref{ineqric} since $\mbox{Ric}_{\mathbb S^n} =(n-1) g_{\mathbb S^n}$. For the equality case, we first observe that $\lambda_{1,1}'' (\mathbb S^n)=2(n-1)$ and that the associated 1-eigenforms correspond to Killing vector fields. Now, it suffices to apply Theorem \ref{genineq} observing that the equation $\Delta_{HR}B=2(n-1) B$ is equivalent to $ \Delta_{HR}\delta \psi=2(n-1) \delta \psi$ (since $\mathbb S^n$ admits no nontrivial harmonic 1-forms), which tells us that $\delta \psi$ is a Killing vector field. \end{proof}
If $r<\frac \pi 2$, then the equality holds if and only if $B=0$ and $q$ is constant. When $r=\frac \pi 2$ (i.e for a hemisphere),   the equality  holds in \eqref{sphericalcap} if and only if $A^{\sharp}$ is a Killing vector field which vanishes at $x_0$  and $\vert A^{\sharp}\vert ^2 +q= \lambda_1( H_{A,q} )$. 

 \end{cor}
Indeed, a Killing vector field is tangent along $\partial C_r(x_0)$ if and only if it vanishes at $x_0$. Since $\partial C_r(x_0)$ is totally umbilical, the condition $S(A^{\sharp})=0$ implies that $A^{\sharp}=0$ unless $S=0$ which only occurs when $r=\frac \pi 2$.  

\smallskip
 In particular, for a magnetic Laplacian $\Delta_A=H_{A,0}$ on $\mathbb S^n$, the inequality \eqref{ineqsphere} reads
 \begin{equation}\label{ineqsphere q=0}
 \lambda_1( H_{A,0} )\le \frac 1{2(n-1)\sigma_n} \Vert B\Vert^2.
 \end{equation}  
 
 Notice that when $n$ is even,  there is no nonzero vector field of constant length on $\mathbb S^n$. Therefore, in the even dimensional case, 
 the equality in \eqref{ineqsphere q=0} holds if and only if $A^{\sharp}=0$ (or, equivalently, $B=0$). If $n$ is odd, then  the equality in \eqref{ineqsphere q=0} implies that $A^{\sharp}$ is proportional to the vector field $J(x)=(-x_2,x_1,\cdots, -x_{n+1}, x_n)$ which is the only Killing vector field of constant length, up to a dilation (see \cite{DS}).

In dimension 2, the inequality \eqref{ineqsphere q=0} (i.e. $\lambda_1( H_{A,0} )\le \frac 1{8\pi} \Vert B\Vert^2 $)
improves the upper bound obtained by Besson, Colbois and Courtois in \cite{BCC}.

%%%

\subsubsection{Closed hypersurfaces}

We now assume that $M$ is a closed, immersed hypersurface of a Riemannian manifold $M'$. At any point $x\in M$ denote the principal curvatures of $M$ (eigenvalues of the shape operator) by
$
k_1(x),\dots,k_n(x).
$
Let $I_p$ denote the set  of $p$-multi-indices
$$
I_p=\{(j_1,\dots,j_p): 1\leq j_1\leq \cdots\leq j_p\leq n\},
$$
and , for each $\alpha=(j_1,\dots,j_p)\in I_p$, consider the corresponding $p$-curvature 
$$
K_{\alpha}(x)=k_{j_1}(x)+\dots+k_{j_p}(x).
$$ 
Set $\star\alpha=\{1,\dots,p\}\setminus\{j_1,\dots,j_p\}$, and moreover
$$
\twosystem
{\beta_p(x)=\frac{1}{p(n-p)}\inf_{\alpha\in I_p}K_{\alpha}(x)K_{\star\alpha}(x)}
{\beta_p(\Sigma)=\inf_{x\in\Sigma}\beta_p(x)}
$$
We then have the following  lower bound (see Theorem 7 in \cite{Sa1}):

\begin{thm}  Let $M^n$ be a closed immersed hypersurface of the Riemannian manifold $M'^{n+1}$ having curvature operator bounded below by $\gamma_{M'}\in\reals$. 
Then we have the following lower bound 
$$
\lambda_{1,p}(M)\geq p(n-p+1)(\gamma_{M'}+\beta_p(M)).
$$
Equality holds for geodesic spheres in constant curvature spaces. In particular
$$
\lambda_{1,1}''(M)\geq \max\{2(n-1)(\gamma_{M'}+\beta_2(M)), n(\gamma_{M'}+\beta_1(M)\}.
$$
\end{thm}

We say that $M$ is {\it $p$-convex} if all $p$-curvatures are non-negative; that is, $K_{\alpha}(x)\geq 0$ for all $\alpha\in I_p$ and for all $x\in M$. Clearly if $M$ is $p$-convex then it is $q$-convex for all $q\geq p$. Then, $1$-convexity is the usual convexity assumption and $n$-convex is equivalent to mean convexity.   

\medskip

Note that we could have a positive lower bound even when the  curvature of $M'$ is negative; it is enough to assume that the $p$-curvatures $K_{\alpha}$ are positive enough. For example, for a $2$-convex hypersurface in hyperbolic space ${\bf H}^{n+1}$, (where $\gamma_{M'}=1$) with $2$-curvatures $K_{\alpha}(x)$ uniformly bounded below by $c>2$, elementary algebra shows that $\beta_2(M)\geq c^2/4$ hence 
$$
\lambda_{1,1}''\geq 2(n-1)(\frac{c^2}4-1)>0
$$
On the other hand, if $M^n$ is a $2$-convex hypersurface of the sphere $\sphere{n+1}$ then $\beta_2(M)\geq 0$ and therefore
$$
\lambda_{1,1}''\geq 2(n-1).
$$

\medskip

We finally remark the following estimate by P. Guerini (\cite{Gue}): if $M$ is a convex hypersurface of $\real n$ then
$$
\lambda_{1,p}(M)\geq\dfrac{p}{2e^3}\cdot\frac{1}{{\rm diam}(M)^2}.
$$

%%%

\subsubsection{Convex domains in Euclidean space} Assume now that $M$ is a convex domain in $\real n$. Then we know from \cite{Sa2} that  for all $p=1,\dots,n$
one has $\lambda_{1,p}=\lambda_{1,p}'$; in particular :
$$
\lambda_{1,1}''=\lambda_{1,2}.
$$
Theorem 1.1 in \cite{Sa2} states that, for all $p=1,\dots,n$:
$$
\dfrac{a_{n,p}}{D_p^2}\leq\lambda_{1,p}\leq \dfrac{a'_{n,p}}{D_p^2}
$$
for explicit constants $a_{n,p},a'_{n,p}$. 
Here $D_p$ is the $p$-th largest principal axis of the ellipsoid of maximal volume included in $M$, also called {\it John ellipsoid} of $M$.  In particular, 
$$
\lambda_{1,1}''\geq \dfrac{4}{n^3 D_2^2}.
$$
Accordingly we have an upper bound for the spectrum of the magnetic Laplacian:
$$
\lambda_1(H_{A,0})\leq \frac{1}{\vert M\vert}\left( \frac{\Vert B\Vert^2}{\lambda_{1,1}''(M) } + \int_{M} q v_g \right)\leq\dfrac{n^3\norm{B}^2D_2^2}{4\abs{M}}+\dfrac{1}{\abs{M}}\int_{M} q v_g
$$
For example, assume that $q=0$ and 
$$
\dfrac{1}{\abs{M}}\int_M\norm{B}^2\leq c.
$$
We then see
$$
\lambda_1(H_{A,0})\leq \frac{cn^3}{4}D_2^2
$$

%%%

\subsubsection{Other estimates}

We refer to \cite{GS} for a lower bound of $\lambda_{1,1}''$ of any compact manifold with boundary $\Omega$, in terms of a lower bound $\gamma\in\reals$ of the eigenvalues of the curvature operator of $\Omega$, and the $2$-curvatures of $\bd\Omega$: if the $2$ curvatures are large enough, then the lower bound is positive (see Theorem 3.3 in \cite{GS}). We also remark that in certain cases it is possible to estimate from below the gap $\lambda_{1,p}-\lambda_{1,0}$ between the first eigenvalue for $p$-forms (absolute boundary conditions) and the first eigenvalue on functions (Neumann conditions). For example, for convex domains in $\sphere n$ one has, for $p=2,\dots, \frac n2$:
$$
\lambda_{1,p}\geq \lambda_{1,0}+(p-1)(n-p)
$$
which reduces to an equality when $\Omega$ is the hemisphere. In particular,
$$
\lambda_{1,1}''\geq \lambda_{1,0}+n-2
$$
which often improves the bound $\lambda_{1,1}''\geq 2(n-1)$ considered in Corollary \ref{cor sphere} above : in fact, $\lambda_{1,0}$ is the first positive eigenvalue for the Neumann Laplacian acting on functions, which can be very large (for example,  for small geodesic balls).

%%%

\section{Upper bounds for the second eigenvalue of $H_{A,q}$} \label{second}

Let us first give the proof of Theorem \ref{volconf}: it is a consequence of the observation (\ref{comparaison}) and of a previous result of El Soufi and Ilias \cite{EI2}. By (\ref{comparaison}), we have
$$
\lambda_2(H_{A,q}) \le \lambda_2(H(0,\vert A\vert^2 +q)
$$
which corresponds to the usual Laplacian $\Delta$ on $(M,g)$ with the potential $\vert A\vert^2 +q$, and $A=\delta \psi +h$ as in (\ref{decomposition}). By \cite{EI2}, for any scalar potential $W$ on $M$ one has  
$$\lambda_2(\Delta +W)\le n\left(\frac {V_c(M)}{\vert M\vert}\right)^{\frac2n} + \frac {1}{\vert M\vert}\int_MW v_g, $$ 
where $V_c(M)$ is the Li-Yau conformal volume of the Riemannian manifold $M$. In our situation, $W=\vert \delta \psi +h\vert^2 +q$ and we have already seen that
$$
\frac{1}{\vert M\vert}\int_M (\vert \delta \psi +h\vert^2 +q)v_g \le \Gamma(M,A,q):=\frac{1}{\vert M\vert}\left( d(h,{\cal L}_{\bf Z})^2 +\frac{\Vert B\Vert^2}{\lambda_{1,1}''(M) } + \int_{M} qv_g \right)
$$
which allows to conclude.

\medskip 
As for the first eigenvalue, we have a lot of consequences of this result in specific situations. For example, the conformal volume of the sphere $\mathbb S^n$ endowed with the conformal class of its standard metric $g_s$ is equal to the volume $\sigma_n=\vert\mathbb S^n\vert_{g_s}$ of the standard metric. %$V_c(\mathbb S^n)=\sigma_n$, the volume of the unit Euclidean sphere. 
Hence, any domain $\Omega\subset\mathbb S^n$, endowed with a metric conformal to the standard one  will satisfies $V_c(\Omega)\le \sigma_n$. 

 \begin {cor} \label{confeuc}
 Let $H_{A,q} $ be a magnetic Schr\"{o}dinger operator on a bounded domain $\Omega\subset \reals^n$, endowed with a Riemannian metric $g$ conformally equivalent to the Euclidean metric.  One has, under  Neumann boundary conditions,
 \begin{equation}\label{confeuc1}
 \lambda_2(H_{A,q})  \le n\left(\frac {\sigma_n}{\vert \Omega\vert_g}\right)^{\frac2n} + \frac 1{\vert \Omega\vert_g} \left(\frac {\Vert B\Vert^2}{ \lambda_{1,1}'' (\Omega)}  + d(h,{\cal L}_{\bf Z})^2+\int_\Omega q v_g\right).
 \end{equation}  
\end {cor}
This corollary applies of course when $\Omega$ is a domain of the Euclidean space, the hyperbolic space, and the sphere. Note that the equality holds in \eqref{confeuc1} when $g$ is the spherical metric, $A=0$, $q=0$  and  $\Omega$ is a ball whose Euclidean radius tends to infinity.

For a compact orientable surface $M$ of genus $\gamma$, one has (see \cite{LY})
\begin{equation}\label{confsurf}
{V_c(M)}\le 4\pi\left[{\frac{\gamma+3}{2}}\right]
 \end{equation} 
 and
 $$ \lambda_{1,1}''(M)=\mu (M)$$
 where $[ ]$ stands for the floor function and $\mu(M)$ is the first positive  eigenvalue of the Laplacian of $M$ acting on functions, with Neumann boundary conditions if $\partial M\ne \emptyset$.  Thus, the inequality (\ref{confvol_ineq}) leads to the following:
 
 \begin {cor} \label{cor surf}
Let $H_{A,q} $  be a magnetic Schr\"{o}dinger operator  on a domain $\Omega $ of a compact orientable Riemannian surface $M$ of genus $\gamma$. One has
\begin{equation}\label{ineqsurf}
\lambda_2( H_{A,q} ) \vert \Omega\vert\le  8\pi\left[{\frac{\gamma+3}{2}}\right] +  \frac {\Vert B\Vert^2}{ \mu(\Omega)}  + d(h,{\cal L}_{\bf Z})^2+ \int_{\Omega}qv_g .
\end{equation}   
where $  \mu(\Omega)$ is the first positive  eigenvalue of the Laplacian on functions, with Neumann b. c. if $\Omega \subsetneq M$.
%Moreover, the equality holds in \eqref{ineqcross}, if and only if $A_c$ is a Killing vector field of $\mathbb S^n$ and $\vert A_c\vert ^2 +q= \lambda_1( H_{A,q} )$.
 \end{cor}

The following corollary extends Hersch's inequality
  \begin {cor} \label{cor sphere}
Let $H_{A,q} $  be a magnetic Schr\"{o}dinger operator  on a compact orientable Riemannian surface $M$ of genus zero. One has
\begin{equation}\label{ineqsphere2}
\lambda_2( H_{A,q} ) \vert M\vert\le  8\pi + \frac {\Vert B\Vert^2}{ \mu(M)}  + \int_{M}qv_g .
\end{equation}   
%where $ \mu(\Omega)$ is the first positive  eigenvalue of the Laplacian on functions, with Neumann b. c. if $\Omega \subsetneq M$.
%Moreover, the equality holds in \eqref{ineqcross}, if and only if $A_c$ is a Killing vector field of $\mathbb S^n$ and $\vert A_c\vert ^2 +q= \lambda_1( H_{A,q} )$.
 \end{cor}

In \cite{EI1}, we have proved that if a Riemannian manifold $M$ admits an isometric immersion in a Euclidean space whose components are first eigenfunctions of the Laplacian, then \begin{equation}\label{confvol}
\left(\frac {V_c(M)}{\vert M\vert}\right)^{\frac2n}  =\frac { \mu(M)}n.
 \end{equation} 
%where $\mu(M)$ is the first positive  eigenvalue of the Laplacian of $M$ (acting on functions). 

 In particular, the equality \eqref{confvol} holds for any compact rank one symmetric space. 
Such a space is Einstein and satisfies $H^1(M,\reals)=\{0\}$.  Thus, combining with Lemma \ref{lem ric},  we get the following

\begin {cor} \label{cor CROSS}
Let $H_{A,q} $  be a magnetic Schr\"{o}dinger operator  on a domain $\Omega $ (with convex boundary if $\Omega \subsetneq M$) of a compact rank one symmetric space $M$ of (real) dimension  $n$. 
One has
\begin{equation}\label{ineqcross}
\lambda_2( H_{A,q} )\le \mu(M)\left(\frac {\vert M\vert}{\vert \Omega\vert}\right)^{\frac2n}+ \frac {1}{\vert \Omega\vert}\left(\frac {\Vert B\Vert^2}{2 c_M}  +\int_{\Omega} qv_g\right) 
\end{equation}   
where $c_M$ is the Ricci curvature constant  of $M$ and $B=dA$ is the magnetic field.
%Moreover, the equality holds in \eqref{ineqcross}, if and only if $A_c$ is a Killing vector field of $\mathbb S^n$ and $\vert A_c\vert ^2 +q= \lambda_1( H_{A,q} )$.
 \end{cor}

%For example, if  $H_{A,q} $  is a magnetic Schr\"{o}dinger operator on the standard sphere $\mathbb S^n$, then 
%\begin{equation}\label{ineqsphere2}\lambda_2( H_{A,q} )\le n+ \frac 1{\sigma_n}\left(\frac {\Vert B\Vert^2}{2 (n-1)}  +\int_{\mathbb S^n} qv_g\right) \end{equation} 

When $M$ is a closed immersed submanifold in a Riemannian space form of curvature $\kappa = -1, \ 0, \ +1$ it was established (\cite{EI3, EI4})  the following relationship between the second eigenvalue of a scalar Schr\"{o}dinger operator $\Delta+W$ and the  $L^2$-norm mean curvature $h_M$ of $M$ : 
$$\lambda_2(\Delta +W)\le  \frac {1}{\vert M\vert}\int_M\left( n \vert h_M \vert ^2 +n\ \kappa +W\right) v_g $$ 
This inequality is known as Reilly inequality when $\kappa =0$ and $W=0$.

\smallskip
The same arguments as before enable us to obtain the following 
\begin {cor} \label{cor reilly}
Let $H_{A,q} $  be a magnetic Schr\"{o}dinger operator  on a closed immersed submanifold $M$ of a space-form of curvature $\kappa = -1, \ 0, \ +1$. 
One has
\begin{equation}\label{ineqreilly}
\lambda_2( H_{A,q} ){\vert M\vert} \le   \int_{M} \left( n \vert h_M \vert ^2 +n\ \kappa\right)v_g +\frac 1{  \lambda_{1,1}''(M)} \Vert B\Vert^2 + d(h,{\cal L}_{\bf Z})^2+ \int_{M}  qv_g.
\end{equation}   
%In particular, if  $H^1(M,\reals)=\{0\}$, then 
%\begin{equation}\label{ineqreilly1} \lambda_2( H_{A,q} )\le  \frac 1{\vert M\vert} \int_{M} \left( n \vert h_M \vert ^2 +n\ \kappa+\frac 1{  \lambda_{1,1}''(M)} \vert B\vert^2 +q\right) v_g.\end{equation}  

 \end{cor}
 
 \section{Upper bounds for higher order eigenvalues of $H_{A,q}$} \label{third}

In order to prove Theorem \ref{main1}, we use again the relation (\ref{comparaison})   
$$
\lambda_k(H_{A,q}) \le \lambda_k(H_{0,\vert A\vert^2 +q}).
$$

In order to prove the inequality (\ref{conformal}), we use the recent \cite{GNS} Theorem 1.1: for a scalar Schr\"{o}dinger operator $\Delta+W$ on a compact Riemannian manifold without boundary. From this result, we deduce that
$$
\lambda_k(\Delta+W)\le  \frac{1}{\vert M\vert}\int_M Wv_g+ c([g])  \left(\frac{k}{|M|}\right)^{\frac 2 n} 
$$
where $c([g])$ is a constant depending only on the conformal class $[g]$ of $g$, and the conclusion follows as before because $W=\vert A\vert^2+q$.

\medskip
In order to prove Inequality (\ref{Ricci}),  one can make use of the estimates obtained by A. Hassannezhad \cite{Ha} for a scalar Schr\"{o}dinger operator $\Delta+W$ on a compact Riemannian manifold: If $\lambda_1(\Delta+W)\ge 0$ (which is in particular the case if $W\ge 0$ as in our situation), then 
 $$
\lambda_k(\Delta+W)\le \frac{c_1}{\vert M\vert}\int_M Wv_g+ c_2 \left(\frac{V([g] )}{|M|}\right)^{\frac 2 n}  +c_3\left(\frac{k-1 }{|M|}\right)^{\frac 2 n} 
$$
 where $c_1$, $c_2$ and $ c_3$ are constants which depend only on the dimension $n$ and $V([g] )$ is the infimum of the volume of $M$ with respect to all Riemannian metrics $g_0$ conformal to $g$ and such that $\mbox{Ric}_{g_0}\ge -(n-1)g_0$. 

In particular, if $\mbox{Ric}_{g}\ge -(n-1)a^2g$ for some $a\ne 0$,  then the metric $g_0=a^2g$ satisfies $\mbox{Ric}_{g_0}\ge -(n-1)g_0$ and $\vert M\vert_{g_0}= a^n \vert M\vert_{g}$. Thus,  $V([g] )\le a^n \vert M\vert_{g}$. 

So, we can conclude by observing that
$$
\int_M W v_g=\int_M(\vert A\vert^2+q)v_g \le \Gamma(M,A,q).
$$

As a corollary, on  a compact orientable surface $M$  of genus $\gamma\ge 2$, every Riemannian metric $g$ is conformal to a hyperbolic metric $g_0$ which implies $V([g] )\le \vert M\vert_{g_0}= 4\pi(\gamma-1)$.  The same observations as before lead to the following

 \begin {cor} \label{asma}
 Let $H_{A,q} $ be a magnetic Schr\"{o}dinger operator on a compact orientable surface of genus $\gamma$, then 
$$
\lambda_k(H_{A,q} ) \vert M\vert\le  a k  +b\gamma +c\Gamma(M,A,q).
$$
  where $a$, $b$ and $ c$ are universal constants.
 
\end{cor}
Let us now consider a magnetic Schr\"{o}dinger operator $H_{A,q} $ on a bounded domain of an Euclidean space (here, as precised before, we consider Neumann condition on the boundary). The following estimates for the sum of eigenvalues (generalizing that of Kr\"oger for $H_{0,0}$), for the Riesz means and for the trace of the magnetic heat kernel (generalizing that of Kac for $H_{0,0}$) are consequences of the considerations above and the estimates obtained in \cite{EHIS}. For convenience and as before, we use the notation $\Gamma(\Omega,A,q):=\frac{1}{\vert \Omega\vert}\left( d(h,{\cal L}_{\bf Z})^2 +\frac{\Vert B\Vert^2}{\lambda_{1,1}''(\Omega) } + \int_{\Omega} q v_g \right)$

 \begin {thm} \label{lambdak}
 Let $H_{A,q} $ be a magnetic Schr\"{o}dinger operator on a bounded domain $\Omega$ of $\R^n$. One has
 
 \noindent (1) For all $z\in\R$,
\begin{equation}\label{Lieb-Th_gen}
 \sum_{j\ge 1} \left(z- \lambda_j(H_{A,q} ) \right)_+ \ge \frac{2\ \vert\Omega\vert}{n+2} \, {\mathcal W}_n^{-\frac n2} \left( z- \Gamma(\Omega,A,q)   \right)_+^{1+\frac n 2},
\end{equation}
where $ {\mathcal W}_n={4\pi^2}/{\omega_n^{\frac 2 n} }$ is the Weyl constant.

\smallskip
\noindent (2) For all $k\ge 1$, %which extends the well-known Kröger inequality 
\begin{equation}\label{EHIS}
\frac 1k\sum_{j=1}^k\lambda_j(H_{A,q} )  \le \frac{n}{n+2}  {\mathcal W}_n  \left(\frac{k-1 }{|\Omega|}\right)^{\frac 2 n} 
+\Gamma(\Omega,A,q)  .
\end{equation} 
and, if $\sum_{j=1}^k\lambda_j(\Delta+q )\ge 0$, 
\begin{equation}\label{EHIS1}
\lambda_k(H_{A,q} )  \le \max\left( 2\left({n+2}\right)^{\frac 2 n}   {\mathcal W}_n \left(\frac{k-1 }{|\Omega|}\right)^{\frac 2 n} 
\ , \  2\Gamma(\Omega,A,q)  \right).
\end{equation} 

\noindent (3) For all $t>0$,
\begin{equation}\label{heat_gen}
 \sum_{j\ge1} e^{- t\lambda_j(H_{A,q}) } \ge  \frac {\vert\Omega\vert}{\left(4\pi t\right)^{\frac n2} }e^{- t\Gamma(\Omega,A,q)}.   
\end{equation}

%\noindent (4) For all $k\ge 1$, 

\end{thm}

\begin{proof}
Taking $A=\delta\psi+h$.  As we have seen before, 

\begin{equation}\label{comparison1}
\lambda_k(H_{A,q} )\le \lambda_k(\Delta +\vert\delta\psi+h\vert^2 +q ).
\end{equation} 
 In \cite{EHIS}, the authors obtained estimates for the eigenvalues, their Riesz means, their sum  and the  heat trace of a general elliptic operator. For a scalar Schr\"{o}dinger operator $\Delta+W$ on a bounded Euclidean domain $\Omega\subset \R^n$, these estimates take the following form :

\smallskip

\noindent (1) For all $z\in\R$,
\begin{equation}\label{schrod1}
 \sum_{j\ge 1} \left(z- \lambda_j(\Delta +W) \right)_+ \ge  \frac{2\ \vert\Omega\vert}{n+2} \, {\mathcal W}_n^{-\frac n2}  \left( z- \frac1{\vert\Omega\vert}\int_\Omega W dx    \right)_+^{1+\frac n 2},
\end{equation}

\noindent (2) For all $k\ge 1$, 
\begin{equation}\label{schrod2}
\frac 1k\sum_{j=1}^k\lambda_j(\Delta +W)  \le  \frac{n}{n+2}  {\mathcal W}_n  \left(\frac{k-1 }{|\Omega|}\right)^{\frac 2 n} +\frac1{\vert\Omega\vert}\int_\Omega  W dx.  
\end{equation} 
and if $\sum_{j=1}^k\lambda_j(\Delta +W) \ge 0$, then
\begin{equation}\label{schrod3}
\lambda_k({\Delta +W} )  \le \max\left( 2\left({n+2}\right)^{\frac 2 n}   {\mathcal W}_n \left(\frac{k-1 }{|\Omega|}\right)^{\frac 2 n} 
\ , \  \frac2{\vert\Omega\vert}\int_\Omega  W dx\right).
\end{equation} 

\noindent (3) For all $t>0$,
\begin{equation}\label{schrod4}
 \sum_{j\ge1} e^{- t\lambda_j(\Delta +W) } \ge  \frac {\vert\Omega\vert}{\left(4\pi t\right)^{\frac n2} }e^{- \frac t{\vert\Omega\vert}\int_\Omega  Wdx}.   
\end{equation}

To conclude the proof, we simply apply these inequalities to the Schr\"{o}dinger operator $\Delta + W$ with $W=\vert\delta\psi+h\vert^2 +q $ and observe that,  using the same arguments as before, 
$$\frac{1}{\vert\Omega\vert}\int_\Omega  W dx \leq  \Gamma(\Omega,A,q) .$$
 
\end{proof}

 Estimates such as \eqref{schrod1} $\dots$  \eqref{schrod4} are also available in   \cite{EHIS} for a bounded domain $\Omega$ of a Riemannian manifold $M$. However, in this case the constants which involve the geometry of $\Omega$ are less explicit than in the Euclidean case. Therefore, we can deduce that there exist constants $c_1(\Omega), \cdots, c_4(\Omega)$, depending only on $\Omega$ such that,  
for all $z\in\R$, $k\ge 1$ and $t>0$
\begin{equation}\label{Lieb-Th_genriem}
 \sum_{j\ge 1} \left(z- \lambda_j(H_{A,q} ) \right)_+ \ge c_1(\Omega) \left( z- \Lambda(A,q)    \right)_+^{1+\frac n 2},
\end{equation}

\begin{equation}\label{EHISriem}
\frac 1k\sum_{j=1}^k\lambda_j(H_{A,q} )  \le c_2(\Omega)  \left(\frac{k-1 }{|\Omega|}\right)^{\frac 2 n} 
+ \Gamma(\Omega,A,q) ,
\end{equation} 
anf if $\sum_{j=1}^k\lambda_j(\Delta+{q} ) \ge 0$, then
\begin{equation}\label{EHIS1riem}
\lambda_k(H_{A,q} )  \le \max\left( c_3(\Omega) \left(\frac{k-1 }{|\Omega|}\right)^{\frac 2 n} 
\ , \  2 \Gamma(\Omega,A,q)  \right),
\end{equation}

\begin{equation}\label{heat_genriem}
 \sum_{j\ge1} e^{- t\lambda_j(H_{A,q}) } \ge  \frac {c_4(\Omega)}{t^{\frac n2} }e^{- t \Gamma(\Omega,A,q)}.   
\end{equation}

\addcontentsline{toc}{chapter}{Bibliography}
%\listoffigures
%\nocite{*} %Permet d'afficher toute la bibliographie sans forcément avoir fait référence dans le document
\bibliographystyle{plain}
\bibliography{CEIS11.bbl}

\begin{thebibliography}{10}

\bibitem{AFNN}
L.~Abatangelo, V.~Felli, B.~Noris, and M.~Nys.
\newblock Sharp boundary behavior of eigenvalues for {A}haronov-{B}ohm
  operators with varying poles.
\newblock {\em arXiv:1605.09569}, 2016.

\bibitem{Be}
A.~Besse.
\newblock {\em Einstein manifolds}.
\newblock Springer, 1987.

\bibitem{BCC}
G{\'e}rard Besson, Bruno Colbois, and Gilles Courtois.
\newblock Sur la multiplicit\'e de la premi\`ere valeur propre de l'op\'erateur
  de {S}chr\"odinger avec champ magn\'etique sur la sph\`ere {$S^2$}.
\newblock {\em Trans. Amer. Math. Soc.}, 350(1):331--345, 1998.

\bibitem{BH}
V.~Bonnaillie-No\"{e}l and B.~Helffer.
\newblock Nodal and spectral minimal partitions -- {T}he state of the art in
  2015.
\newblock {\em arXiv:1506.07249}, 2015.

\bibitem{CS}
B.~Colbois and A.~Savo.
\newblock Lower bounds for the first eigenvalue of the magnetic {L}aplacian.
\newblock {\em Preprint}, 2017.

\bibitem{CM}
Bruno Colbois and Daniel Maerten.
\newblock Eigenvalues estimate for the {N}eumann problem of a bounded domain.
\newblock {\em J. Geom. Anal.}, 18(4):1022--1032, 2008.

\bibitem{CMT}
Ana-Bela Cruzeiro, Paul Malliavin, and Setsuo Taniguchi.
\newblock Ground state estimations in gauge theory.
\newblock {\em Bull. Sci. Math.}, 125(6-7):623--640, 2001.
\newblock Rencontre Franco-Japonaise de Probabilit{\'e}s (Paris, 2000).

\bibitem{DS}
Antonio~J. Di~Scala.
\newblock Killing vector fields of constant length on compact hypersurfaces.
\newblock {\em Geom. Dedicata}, 175:403--406, 2015.

\bibitem{ELMP}
M.~Egidi, S.~Liu, F.~M\"{u}nch, and N.~Peyerimhoff.
\newblock Ricci curvature and eigenvalue estimates for the magnetic {L}aplacian
  on manifolds.
\newblock {\em arXiv:1608.01955}, 2016.

\bibitem{EI1}
A.~El~Soufi and S.~Ilias.
\newblock Immersions minimales, premi\`ere valeur propre du laplacien et volume
  conforme.
\newblock {\em Math. Ann.}, 275(2):257--267, 1986.

\bibitem{EI2}
A.~El~Soufi and S.~Ilias.
\newblock Majoration de la seconde valeur propre d'un op\'erateur de
  {S}chr\"odinger sur une vari\'et\'e compacte et applications.
\newblock {\em J. Funct. Anal.}, 103(2):294--316, 1992.

\bibitem{EI3}
A.~El~Soufi and S.~Ilias.
\newblock Une in\'egalit\'e du type ``{R}eilly'' pour les sous-vari\'et\'es de
  l'espace hyperbolique.
\newblock {\em Comment. Math. Helv.}, 67(2):167--181, 1992.

\bibitem{EHIS}
A.~El~Soufi, Harrell~E. M., S.~Ilias, and J.~Stubbe.
\newblock On sums of eigenvalues of elliptic operators on manifolds.
\newblock {\em J. Spectr. Theory, to appear}.

\bibitem{EI4}
Ahmad El~Soufi and Sa{\"{\i}}d Ilias.
\newblock Second eigenvalue of {S}chr\"odinger operators and mean curvature.
\newblock {\em Comm. Math. Phys.}, 208(3):761--770, 2000.

\bibitem{Er1}
L.~Erd\"{o}s.
\newblock Rayleigh-type isoperimetric inequality with a homogeneous magnetic
  field.
\newblock {\em Calc. Var. Partial Differential Equations}, 4:283--292, 1996.

\bibitem{Er2}
L.~Erd\"{o}s.
\newblock Spectral shift and multiplicity of the first eigenvalue of the
  magnetic {S}chrödinger operator in two dimensions.
\newblock {\em Ann. Inst. Fourier}, 52:1833--1874, 2002.

\bibitem{FH2}
S{\o}ren Fournais and Bernard Helffer.
\newblock On the isoperimetric inequality for the ground state energy of the
  {N}eumann magnetic {L}aplacian.
\newblock {\em In preparation}.

\bibitem{FH1}
S{\o}ren Fournais and Bernard Helffer.
\newblock {\em Spectral methods in surface superconductivity}.
\newblock Progress in Nonlinear Differential Equations and their Applications,
  77. Birkh\"auser Boston Inc., Boston, MA, 2010.

\bibitem{GM}
S.~Gallot and D.~Meyer.
\newblock Op\'erateur de courbure et laplacien des formes diff\'erentielles
  d'une vari\'et\'e riemannienne.
\newblock {\em J. Math. Pures Appl. (9)}, 54(3):259--284, 1975.

\bibitem{GNS}
A.~Grigor'yan, N.~Nadirashvili, and Y.~Sire.
\newblock A lower bound for the number of negative eigenvalues of
  {S}chr\"odinger operators.
\newblock {\em J. Differential Geom.}, 102:395--408, 2016.

\bibitem{Gue}
P.~Guerini.
\newblock Spectre du laplacien de {H}odge-de {R}ham: estim\'ees sur les
  vari\'et\'es convexes.
\newblock {\em Bull. London Math. Soc.}, 36:88--94, 2004.

\bibitem{GS}
P.~Guerini and A.~Savo.
\newblock Eigenvalue and gap estimates for the laplacian acting on p -forms.
\newblock {\em Trans. Amer. Math. Soc.}, 256:319--344, 2004.

\bibitem{Ha}
Asma Hassannezhad.
\newblock Eigenvalues of perturbed {L}aplace operators on compact manifolds.
\newblock {\em Pacific J. Math.}, 264(2):333--354, 2013.

\bibitem{HHHO}
B.~Helffer, M.~Hoffmann-Ostenhof, T.~Hoffmann-Ostenhof, and M.~P. Owen.
\newblock Nodal sets for groundstates of {S}chr\"odinger operators with zero
  magnetic field in non-simply connected domains.
\newblock {\em Comm. Math. Phys.}, 202:629--649, 1999.

\bibitem{LLPP}
C.~Lange, S.~Liu, N.~Peyerimhoff, and O.~Post.
\newblock Frustration index and {C}heeger inequalities for discrete and
  continuous magnetic {L}aplacian.
\newblock {\em Calc. Var. Partial Differential Equations}, 54:4165--4196, 2015.

\bibitem{LY}
Peter Li and Shing~Tung Yau.
\newblock A new conformal invariant and its applications to the {W}illmore
  conjecture and the first eigenvalue of compact surfaces.
\newblock {\em Invent. Math.}, 69(2):269--291, 1982.

\bibitem{NNT}
B.~Noris, M.~Nys, and S.~Terracini.
\newblock On the eigenvalues of {A}haronov-{B}ohm operators with varying poles:
  pole approaching the boundary of the domain.
\newblock {\em Comm. Math. Phys.}, 339:1101--1146, 2015.

\bibitem{Sa2}
A.~Savo.
\newblock Hodge-{L}aplace eigenvalues of convex bodies.
\newblock {\em Trans. Amer. Math. Soc.}, 363:1789--1804, 2011.

\bibitem{Sa1}
A.~Savo.
\newblock The {B}ochner formula for isometric immersions.
\newblock {\em Pacific J. Math.}, 272:395--422, 2014.

\bibitem{Sc}
G.~Schwarz.
\newblock {\em Hodge decomposition—a method for solving boundary value
  problems}.
\newblock Lecture Notes in Mathematics, 1607. Springer, 1995.

\bibitem{Sh}
I.~Shigekawa.
\newblock Eigenvalue problems for the {S}chr\"{o}dinger operator with the
  magnetic field on a compact {R}iemannian manifold.
\newblock {\em J. Funct. Anal.}, 75:92--127, 1987.

\end{thebibliography}

\bigskip

\normalsize 
\noindent Bruno Colbois

\noindent Universit\'e de Neuch\^atel, Institut de Math\'ematiques \\
Rue Emile Argand 11\\
 CH-2000, Neuch\^atel, Suisse

\noindent bruno.colbois@unine.ch

\bigskip
\noindent
Ahmad El Soufi and Sa\"{\i}d Ilias

\noindent
 Universit\'e Fran{\c c}ois Rabelais de Tours, Laboratoire de Math\'ematiques
et Physique Th\'eorique\\
UMR-CNRS 6083, Parc de Grandmont, 37200
Tours, France

\noindent  
ilias@univ-tours.fr

\bigskip

\noindent Alessandro Savo

\noindent Dipartimento SBAI, Sezione di Matematica \\
Sapienza Universit\`a di Roma,  
Via Antonio Scarpa 16\\
00161 Roma, Italy

\noindent alessandro.savo@uniroma1.it

\end{document}